\documentclass[11pt]{amsart}

\usepackage{amsfonts}
\usepackage{hyperref}
\usepackage{graphicx}
\usepackage{amsbsy}
\usepackage[all]{xy}
\usepackage{amsmath}
\usepackage{amssymb}
\usepackage{amsthm, amscd, cite} 
\usepackage{enumitem}
\usepackage{mathtools}
\usepackage[mathscr]{eucal}
\numberwithin{equation}{section}

\newtheorem{theorem}{\bf Theorem}[section]
\newtheorem{main}{\bf Theorem}

\newtheorem{lemma}[theorem]{\bf Lemma}
\newtheorem{proposition}[theorem]{\bf Proposition}

\theoremstyle{definition}

\numberwithin{equation}{section}
\makeatletter
\let\c@theorem\c@equation
\makeatother

\newtheorem*{namedtheorem}{\theoremname}
\newcommand{\theoremname}{testing}

\renewcommand{\leq}{\leqslant}
\renewcommand{\geq}{\geqslant}
\newcommand{\norm}{\trianglelefteq}

\newcommand{\gen}[1]{\langle #1 \rangle}
\newcommand{\ol}{\overline}
\newcommand{\ov}{\bar}

\newcommand{\F}{\mathcal{F}}

\renewcommand{\L}{\mathcal{L}}

\renewcommand{\O}{\mathcal{O}}

\newcommand{\T}{\mathcal{T}}

\newcommand{\Z}{\mathcal{Z}}
\renewcommand{\phi}{\varphi}

\newcommand{\id}{\operatorname{id}}
\newcommand{\incl}{\operatorname{incl}}

\newcommand{\Hom}{\operatorname{Hom}}
\newcommand{\Mor}{\operatorname{Mor}}

\newcommand{\Ob}{\operatorname{Ob}}

\newcommand{\Spin}{\operatorname{Spin}}
\newcommand{\Sol}{\operatorname{Sol}}

\newcommand{\Aut}{\operatorname{Aut}}
\newcommand{\Out}{\operatorname{Out}}
\newcommand{\Outdiag}{\operatorname{Outdiag}}
\newcommand{\Inn}{\operatorname{Inn}}

\newcommand{\Syl}{\operatorname{Syl}}

\renewcommand{\hat}{\widehat}

\newcommand{\hyp}{\mathfrak{hyp}}

\newcommand{\typ}{\textup{typ}}

\title{Extensions of the Benson-Solomon fusion systems}
\author{Ellen Henke}
\email{ellen.henke@abdn.ac.uk}
\address{Institute of Mathematics \\University of Aberdeen \\ Fraser
Noble Building \\ Aberdeen AB15 5LY \\ United Kingdom} 
\author{Justin Lynd}
\email{lynd@louisiana.edu}
\address{Department of Mathematics\\University of Louisiana at Lafayette\\Maxim
Doucet Hall\\Lafayette, LA 70504}

\thanks{
Justin Lynd was partially supported by NSA Young Investigator Grant
H98230-14-1-0312.  This project has received funding from the European Union's
Horizon 2020 research and innovation programme under the Marie
Sk{\l}odowska-Curie grant agreement No. 707758.
}
\dedicatory{To Dave Benson on the occasion of his second 60th birthday}
\keywords{fusion system, linking system, Benson-Solomon fusion system, group extension}
\subjclass[2000]{Primary 20D20, Secondary 20D05}
\date{\today}

\begin{document}
\begin{abstract}
The Benson-Solomon systems comprise the one currently known family of simple
saturated fusion systems at the prime two that do not arise as the fusion
system of any finite group. We determine the automorphism groups and the
possible almost simple extensions of these systems and of their centric linking
systems.
\end{abstract}
\maketitle

\section{Introduction}\label{S:intro}

There is one known family of simple exotic fusion systems at the prime $2$, the
Benson-Solomon systems. They were first predicted by Dave Benson
\cite{Benson1998c} to exist as finite versions of a $2$-local compact group
associated to the $2$-compact group $DI(4)$ of Dwyer and Wilkerson
\cite{DwyerWilkerson1993}. They were later constructed by Levi and Oliver
\cite{LeviOliver2002} and Aschbacher and Chermak \cite{AschbacherChermak2010}.
The purpose of this paper is determine the automorphism groups of the
Benson-Solomon fusion and centric linking systems, and use that information to
determine the fusion systems whose generalized Fitting subsystem is a
Benson-Solomon system.  This information is needed within certain portions of
Aschbacher's program to classify simple fusion systems of component type at the
prime $2$.  In particular, it is needed within an involution centralizer
problem for these systems. Some results on automorphisms of these systems
appear in the standard references \cite{LeviOliver2002, LeviOliver2005,
AschbacherChermak2010}, and part of our aim is to complete the picture.  We now
summarize the main results. Slightly more detailed statements are contained in
the statements of Theorem~\ref{T:outsol} and Theorem~\ref{T:solext}. 

\begin{main}
\label{T:mainint}
Fix an odd prime power $q$, and let $l = v_2(q^2-1)-3$ where $v_2$ is the
$2$-adic valuation. Set $\F_0 = \F_{\Sol}(q)$, a Benson-Solomon fusion system
over the $2$-group $S_0$. Let $\F$ be any almost simple extension of $\F_0$,
namely, any saturated fusion system over a $2$-group $S$ such that $F^*(\F) =
\F_0$. Then 
\begin{enumerate}
\item[(a)] $\Out(\F_0) \cong  C_{2^l}$ is cyclic of order $2^l$, induced by
field automorphisms, and
\item[(b)] $\F_0 = O^2(\F)$, $S$ splits over $S_0$, and the induced map $S/S_0
\to \Out(\F_0)$ is injective. 
\end{enumerate}
Moreover, for each subgroup $A \leq \Out(\F_0)$, there is a unique almost
simple extension $\F$ of $\F_0$ as above, up to isomorphism, such that the map
$S/S_0 \to \Out(\F_0)$ has image $A$. 
\end{main}

The paper proceeds as follows. In Section~\ref{S:autfus}, we recall the various
automorphism groups of fusion and linking systems and the maps between them,
following \cite{AOV2012}. In Section~\ref{S:aut}, we look at automorphisms of
the fusion and linking systems of $\Spin_7(q)$ and of the Benson-Solomon
systems. We show in Theorem~\ref{T:outsol} that the outer automorphism group of
the latter is a cyclic group of field automorphisms of $2$-power order.
Finally, we show in Theorem~\ref{T:solext} that the systems having a
Benson-Solomon generalized Fitting subsystem are uniquely determined by the
outer automorphisms they induce on the fusion system, and that all such
extensions are split. 

All our maps are written on the left. We would like to thank Jesper Grodal, Ran
Levi, and Bob Oliver for helpful conversations. 

\section{Automorphisms of fusion and linking systems}\label{S:autfus}

We refer to \cite{AschbacherKessarOliver2011} for the definition of a saturated
fusion system, and also for the definition of a centric subgroup of a fusion
system. Let $\F$ be a saturated fusion system over the finite $p$-group $S$,
and write $\F^c$ for the collection of $\F$-centric subgroups. Whenever $g$ is
an element of a finite group, we write $c_g$ for the conjugation homomorphism
$x \mapsto { }^g\!x =  gxg^{-1}$ and its restrictions.

\subsection{Background on linking systems}
Whenever $\Delta$ is an overgroup-closed, $\F$-invariant collection of
subgroups of $S$, we have the transporter category $\T_{\Delta}(S)$ with those
objects. This is the full subcategory of the transporter category $\T_S(S)$
where the objects are subgroups of $S$, and morphisms are the transporter sets:
$N_S(P,Q) = \{s \in S \mid sPs^{-1} \leq Q\}$ with composition given by
multiplication in $S$.

A \emph{linking system} associated to $\F$ is a nonempty category $\L$ with
object set $\Delta$, together with functors
\begin{eqnarray}
\label{D:linking}
\T_{\Delta}(S) \xrightarrow{\quad\delta\quad} \L \xrightarrow{\quad\pi\quad} \F.
\end{eqnarray}
The functor $\delta$ is the identity on objects and injective on morphisms,
while $\pi$ is the inclusion on objects and surjective on morphisms. 
Write $\delta_{P,Q}$ for the corresponding injection $N_S(P,Q) \to
\Mor_\L(P,Q)$ on morphisms, write $\delta_P$ for $\delta_{P,P}$, and use
similar notation for $\pi$.

The category and its structural functors are subject to several axioms which
may be found in \cite[Definition~II.4.1]{AschbacherKessarOliver2011}. In
particular, Axiom (B) states that for all objects $P$ and $Q$ in $\L$ and each
$g \in N_S(P,Q)$, we have $\pi_{P,Q}(\delta_{P,Q}(g)) = c_g \in \Hom_\F(P,Q)$.
A \emph{centric linking system} is a linking system with $\Delta = \F^c$. Given
a finite group $G$ with Sylow $p$-subgroup $S$, the canonical centric linking
system for $G$ is the category $\L^c_S(G)$ with objects the $p$-centric
subgroups $P \leq S$ (namely those $P$ whose centralizer satisfies $C_G(P) =
Z(P) \times O_{p'}(C_G(P))$), and with morphisms the orbits of the transporter
set $N_G(P,Q) = \{g \in G \mid gPg^{-1} \leq Q\}$ under the right action of
$O_{p'}(C_G(P))$. 

\subsubsection{Distinguished subgroups and inclusion morphisms} 

The subgroups $\delta_P(P) \leq \Aut_\L(P)$ are called \emph{distinguished
subgroups}. When $P \leq Q$, the morphism $\iota_{P,Q}:=\delta_{P,Q}(1) \in
\Mor_{\L}(P,Q)$ is the \emph{inclusion} of $P$ into $Q$. 

\subsubsection{Axiom (C) for a linking system}\label{SS:axiomc} We will make
use of Axiom (C) for a linking system, which says that for each morphism $\phi
\in \Mor_{\L}(P,Q)$ and element $g \in N_S(P)$, the following identity holds
between morphisms in $\Mor_\L(P,Q)$: 
\[
\phi \circ \delta_P(g) = \delta_Q(\pi(\phi)(g)) \circ \phi.
\]

\subsubsection{Restrictions in linking systems}\label{SS:restrictions}
For each morphism $\psi\in\Mor_\L(P,Q)$, and each $P_0,Q_0\in\Ob(\L)$ such that
$P_0\leq P$, $Q_0\leq Q$, and $\pi(\psi)(P_0)\leq Q_0$, there is a unique
morphism $\psi|_{P_0,Q_0}\in \Mor_\L(P_0,Q_0)$ (the \emph{restriction} of
$\psi$) such that $\psi\circ\iota_{P_0,P} = \iota_{Q_0,Q}\circ
\psi|_{P_0,Q_0}$. See \cite[Proposition~4(b)]{Oliver2010} or
\cite[Proposition~4.3]{AschbacherKessarOliver2011}.

Note that in case $\psi = \delta_{P,Q}(s)$ for some $s \in N_S(P,Q)$, it can be
seen from Axioms (B) and (C) that $\psi|_{P_0,Q_0} = \delta_{P_0,Q_0}(s)$. 

\subsection{Background on automorphisms}
\subsubsection{Automorphisms of fusion systems}
An automorphism of $\F$ is, by definition, determined by its effect on $S$:
define $\Aut(\F)$ to be the subgroup of $\Aut(S)$ consisting of those
automorphisms $\alpha$ which \emph{preserve fusion} in $\F$ in the sense
that the homomorphism given by $\alpha(P) \xrightarrow{\alpha\phi\alpha^{-1}}
\alpha(Q)$ is in $\F$ for each morphism $P \xrightarrow{\phi} Q$ in $\F$. The
automorphisms $\Aut_\F(S)$ of $S$ in $\F$ thus form a normal subgroup of
$\Aut(\F)$, and the quotient $\Aut(\F)/\Aut_\F(S)$ is denoted by $\Out(\F)$. 

\subsubsection{Automorphisms of linking systems}\label{SS:linkaut}
A self-equivalence of $\L$ is said to be \emph{isotypical} if it sends
distinguished subgroups to distinguished subgroups, i.e. $\alpha(\delta_P(P)) =
\delta_{\alpha(P)}(\alpha(P))$ for each object $P$. It sends inclusions to
inclusions provided $\alpha(\iota_{P,Q}) = \iota_{\alpha(P),\alpha(Q)}$
whenever $P \leq Q$. The monoid $\Aut(\L)$ of isotypical
self-equivalences that send inclusions to inclusions is in fact a group of
automorphisms of the category $\L$, and this has been shown to be the most
appropriate group of automorphisms to consider. Note that $\Aut(\L)$
has been denoted by $\Aut_{\typ}^I(\L)$ in \cite{AschbacherKessarOliver2011,
AOV2012} and elsewhere. When $\alpha \in \Aut(\L)$ and $P$ is an object with
$\alpha(P) = P$, we denote by $\alpha_P$ the automorphism of $\Aut_\L(P)$
induced by $\alpha$. 

The group $\Aut_\L(S)$ acts by conjugation on $\L$ in the following way: given
$\gamma \in \Aut_{\L}(S)$, consider the functor $c_\gamma \in
\Aut(\L)$ which is $c_\gamma(P) = \pi(\gamma)(P)$ on objects, and
which sends a morphism $P \xrightarrow{\phi} Q$ in $\L$ to the morphism
$\gamma\phi\gamma^{-1}$ from $c_\gamma(P)$ to $c_\gamma(Q)$ after replacing
$\gamma$ and $\gamma^{-1}$ by the appropriate restrictions (introduced in
\S\ref{SS:restrictions}). Note that when $\gamma = \delta_{S}(s)$ for some $s
\in S$, then $c_\gamma(P)$ is conjugation by $s$ on objects, and
$c_{\gamma}(\phi) = \delta_{Q,{ }^s\!Q}(s)\circ\phi\circ\delta_{{ }^s\!P,P}(s^{-1})$ for
each morphism $\phi \in \Mor_{\L}(P,Q)$ by the remark on distinguished
morphisms in \S\ref{SS:restrictions}. In particular, when $\L = \L_S^c(G)$ for
some finite group $G$, $c_\gamma$ is truly just conjugation by $s$ on
morphisms.

The image of $\Aut_{\L}(S)$ under the map $\gamma \mapsto c_\gamma$ is seen to
be a normal subgroup of $\Aut(\L)$. The outer automorphism group of
$\L$ is 
\[
\Out(\L) := \Aut(\L)/\{c_\gamma \mid \gamma \in \Aut_\L(S)\}. 
\]
We refer to Lemma~1.14(a) and the surrounding discussion in \cite{AOV2012} for
more details. This group is denoted by $\Out_\typ(\L)$ in
\cite{AschbacherKessarOliver2011, AOV2012} and elsewhere. 

\subsubsection{From linking system automorphisms to fusion system
automorphisms}\label{SS:linkfusaut}
There is a group homomorphism 
\begin{eqnarray}
\label{E:muhat}
\widetilde{\mu} \colon \Aut(\L) \longrightarrow \Aut(\F), 
\end{eqnarray}
given by restriction to $S \cong \delta_S(S) \leq \Aut_\L(S)$; see
\cite[Proposition~6]{Oliver2010}. The map $\widetilde{\mu}$ induces a
homomorphism on quotient groups
\[
\mu\colon \Out(\L) \longrightarrow \Out(\F). 
\]
We write $\mu_\L$ (or $\mu_G$ when $\L = \L_S^c(G)$) whenever we wish to make
clear which linking system we are working with; similar remarks hold for
$\widetilde{\mu}$.  As shown in
\cite[Proposition~II.5.12]{AschbacherKessarOliver2011}, $\ker(\mu)$ has an
interesting cohomological interpretation as the first cohomology group of the
center functor $Z_\F$ on the orbit category of $\F$-centric subgroups, and
$\ker(\widetilde{\mu})$ is correspondingly a certain group of normalized
$1$-cocycles for this functor. 

\subsubsection{From group automorphisms to fusion system and linking system
automorphisms}\label{SS:grouplinkaut}
We also need to compare automorphisms of groups with the automorphisms of their
fusion and linking systems. If $G$ is a finite group with Sylow $p$-subgroup
$S$, then each outer automorphism of $G$ is represented by an automorphism that
fixes $S$. This is a consequence of the transitive action of $G$ on its Sylow
subgroups. More precisely, there is an exact sequence:
\[
1 \to Z(G) \xrightarrow{\incl} N_G(S) \xrightarrow{g \mapsto c_g} \Aut(G,S) \to
\Out(G) \to 1. 
\]
where $\Aut(G,S)$ is the subgroup of $\Aut(G)$ consisting of those
automorphisms that leave $S$ invariant.  

For each pair of $p$-centric subgroups $P, Q \leq S$ and each $\alpha \in
\Aut(G,S)$,  $\alpha$ induces an isomorphism $O_{p'}(C_{G}(P)) \to
O_{p'}(C_G(\alpha(P)))$ and a bijection $N_G(P,Q) \to N_G(\alpha(P),
\alpha(Q))$.  Thus, there is a group homomorphism
\[
\widetilde{\kappa}_G \colon \Aut(G,S) \to \Aut(\L_S^c(G))
\]
which sends $\alpha \in \Aut(G,S)$ to the functor which is $\alpha$ on objects,
and also $\alpha$ on morphisms in the way just mentioned.  This map sends the
image of $N_G(S)$ to $\{c_\gamma \mid \gamma \in \Aut_{\L_S^c(G)}(S)\}$, and so
induces a homomorphism
\[
\kappa_G\colon \Out(G) \to \Out(\L_S^c(G))
\]
on outer automorphism groups. 

It is straightforward to check that the restriction to $S$ of any member of
$\Aut(G,S)$ is an automorphism of the fusion system $\F_S(G)$. Indeed, for
every $\alpha\in\Aut(G,S)$, the automorphism $\alpha|_S$ of $\F_S(G)$ is just
the image of $\alpha$ under $\widetilde{\mu}_G\circ \widetilde{\kappa}_G$.

\subsubsection{Summary}
What we will need in our proofs is summarized in the following commutative
diagram, which is an augmented and updated version of the one found in
\cite[p.186]{AschbacherKessarOliver2011}. 

\begin{eqnarray}
\label{E:diagram}
\vcenter{
\xymatrix{
    &     1 \ar[d]      &      1  \ar[d]    &      1 \ar[d] &   \\
  Z(\F) \ar@{=}[d] \ar[r]^{\incl} &  Z(S) \ar[r] \ar[d]^{\delta_S} & \widehat{Z}^1(\O(\F^c), \Z_\F) \ar[d]^{\widetilde{\lambda}} \ar[r] & \varprojlim{\!}^1(\Z_\F) \ar[r] \ar[d]^{\lambda} & 1 \\
  Z(\F) \ar[r] & \Aut_\L(S)  \ar[r] \ar[d]^{\pi_S}  &  \Aut(\L) \ar[r] \ar[d]^{\widetilde{\mu}} &  \Out(\L) \ar[r] \ar[d]^{\mu} & 1 \\
1 \ar[r] & \Aut_\F(S)  \ar[r] \ar[d]  &  \Aut(\F) \ar[r] \ar[d]  & \Out(\F)  \ar[r]\ar[d] & 1\\
    &   1  & 1 & 1 &  \\ 
}
}
\end{eqnarray}

All sequences in this diagram are exact.  Most of this either is shown in the
proof of \cite[Proposition~II.5.12]{AschbacherKessarOliver2011}, or follows
from the above definitions. The first and second rows are exact by this
reference, except that the diagram was not augmented by the maps out of $Z(\F)$
(the center of $\F$); exactness at $Z(S)$ and $\Aut_\L(S)$ is shown by
following the proof there.  Given
\cite[Proposition~II.5.12]{AschbacherKessarOliver2011}, exactness of the last
column is equivalent to the uniqueness of centric linking systems, a result of
Chermak. In all the cases needed in this article, it follows from
\cite[Lemma~3.2]{LeviOliver2002}. The second-to-last column is then exact
by a diagram chase akin to that in a 5-lemma for groups.

\section{Automorphisms}\label{S:aut}

The isomorphism type of the fusion systems of the Benson-Solomon systems
$\F_{\Sol}(q)$, as $q$ ranges over odd prime powers, is dependent only on the
$2$-share of $q^2-1$ by \cite[Theorem~B]{ChermakOliverShpectorov2008}. Since
the centralizer of the center of the Sylow group is the fusion system of
$\Spin_7(q)$, the same holds also for the fusion systems of these groups.  For
this reason, and because of Proposition~\ref{P:autspin} below, it will be
convenient to fix a nonnegative integer $l$, and take $q_l = 5^{2^l}$ for the
sequel. Let ${\bf F}$ be the algebraic closure of the field with five
elements. 

\subsection{Automorphisms of the fusion system of $\Spin_7(q)$} 

Let $\bar{H} = \Spin_7(\mathbf{F})$. Fix a maximal torus $\bar{T}$ of
$\bar{H}$. Thus, $\bar{H}$ is generated by the $\bar{T}$-root groups
$\bar{X}_{\alpha} = \{x_{\alpha}(\lambda)\colon \lambda\in\mathbf{F}\} \cong
(\mathbf{F},+)$, as $\alpha$ ranges over the root system of type $B_3$, and is
subject to the Chevalley relations of \cite[Theorem~1.12.1]{GLS3}.  For any
power $r$ of $5$, we let $\psi_{r}$ denote the standard Frobenius endomorphism
of $\bar{H}$, namely the endomorphism of $\bar{H}$ which acts on the root
groups via $\psi_{r}(x_{\alpha}(\lambda)) = x_{\alpha}(\lambda^{r})$.

Set $H_l = C_{\bar{H}}(\psi_{q_l})$. Thus, $H_l = \Spin_7(q_l)$ since $\bar{H}$
is of universal type (see \cite[Theorem~2.2.6(f)]{GLS3}).  Also,
$T_{\psi_{q_l}} := C_{\bar{T}}(\psi_{q_l})$ is a maximal torus of $H_l$. For
each power $r$ of $5$, the Frobenius endomorphism $\psi_{r}$ of $\bar{H}$ acts
on $H_l$ in the way just mentioned, and it also acts on $T_{\psi_{q_l}}$ by
raising each element to the power $r$. For ease of notation, we denote by
$\psi_{q_l}$ also the automorphism of $H_l$ induced by $\psi_{q_l}$.

We next recall some items from \cite[Lemmas~4.3,\! 4.8,\!
4.9]{AschbacherChermak2010}.  The normalizer $N_{H_l}(T_{\psi_{q_l}})$ contains
a Sylow $2$-subgroup of $H_l$, and $N_{H_l}(T_{\psi_{q_l}})/T_{\psi_{q_l}}$ is
isomorphic to $C_2 \times S_4$, the Weyl group of $B_3$. We may choose such a
Sylow $2$-subgroup $S_l$ of $N_{H_{l}}(T_{\psi_{q_l}})$ to be invariant under
$\psi_{5}$; we fix such a choice for the remainder. Set $k := k_l = l+2$, and
denote by 
\[ 
T_k := T_{\psi_{q_l}}  \cap  S_l \cong C_{2^k} \times C_{2^k} \times C_{2^k}
\]
the $2$-torsion in the maximal torus $T_{\psi_{q_l}}$ of $H_l$.

The automorphism groups of the Chevalley groups were determined by Steinberg
\cite{Steinberg1960}, and in particular, 
\begin{eqnarray} \label{E:autspingrp}
\Out(H_l) = \Outdiag(H_l) \times \Phi \cong C_2 \times C_{2^l}, 
\end{eqnarray}
where $\Phi$ is the group of field automorphisms, and where $\Outdiag(H_l)$ is
the group of outer automorphisms of $H_l$ induced by $N_{\bar{T}}(H_l)$
\cite[Theorem~2.5.1(b)]{GLS3}.  We mention that $S_l$ is normalized by some
element of $N_{\bar{T}}(H_l)-H_l$ of $2$-power order. So we find
representatives of the elements of $\Phi$ and of $\Outdiag(H_l)$ in
$\Aut(H_l,S_l)$. 

We need to be able to compare automorphisms of the group with automorphisms of
the fusion and linking systems, and this has been carried out in full
generality by Broto, M{\o}ller, and Oliver \cite{BrotoMollerOliver2016} for
groups of Lie type.  Let $\F_{\Spin}(q_l)$ and $\L_{\Spin}^c(q_l)$ be the
associated fusion and centric linking systems over $S_l$ of the group $H_l$,
and recall the maps $\mu_{H_l}$ and $\kappa_{H_l}$ from
\S\S\ref{SS:linkfusaut}, \ref{SS:grouplinkaut}. 

\begin{proposition}\label{P:autspin}
The maps $\mu_{H_l}$ and $\kappa_{H_l}$ are isomorphisms, and hence
\[
\Out(\L_{\Spin}(q_l)) \cong \Out(\F_{\Spin}(q_l)) \cong C_2
\times C_{2^l}.
\]
\end{proposition}
\begin{proof}
That $\mu_{H_l}$ is an isomorphism follows from \eqref{E:diagram} and
\cite[Lemma~3.2]{LeviOliver2002}. Also, $\kappa_{H_l}$ is an isomorphism by
\cite[Propositions~5.15, 5.16]{BrotoMollerOliver2016}, using that
$\Outdiag(H_l)$ is a $2$-group. 
\end{proof}

\subsection{Automorphisms of the Benson-Solomon systems}\label{SS:autsol}

We keep the notation from the previous subsection. We denote by $\F :=
\F_{\Sol}(q_l)$ a Benson-Solomon fusion system over the $2$-group $S_l \in
\Syl_2(H_l)$ fixed above, and by $\L := \L^c_{\Sol}(q_l)$ an associated centric
linking system with structural functors $\delta$ and $\pi$. 

\medskip

\emph{For the remainder of Section~\ref{S:aut}, we fix $l \geq 0$, and we set
$H := H_l$ and $S := S_l$.}

\medskip

Observe that $Z(S) \leq T_k$ is of order $2$, and $N_\F(Z(S)) = C_\F(Z(S)) =
\F_{\Spin}(q_l)$ is a fusion system over $S$.  Since $Z(S)$ is contained in
every $\F$-centric subgroup, by Definition~6.1 and Lemma~6.2 of
\cite{BrotoLeviOliver2003}, we may take $N_\L(Z(S)) = C_\L(Z(S))$ for the
centric linking system of $\Spin_7(q_l)$. By the items just referenced,
$C_\L(Z(S))$ is a subcategory of $\L$ with the same objects, and with morphisms
those morphisms $\phi$ in $\L$ such that $\pi(\phi)(z)= z$.  Further,
$C_\L(Z(S))$ has the same inclusion functor $\delta$, and the projection
functor for $C_\L(Z(S))$ is the restriction of $\pi$. (This was also shown in
\cite[Lemma~3.3(a,b)]{LeviOliver2002}.)

Write $\F_z$ for $\F_{\Spin}(q_l)$ and $\L_z$ for $C_\L(Z(S))$ for short.  Each
member of $\Aut(\F)$ fixes $Z(S)$ and so $\Aut(\F)\subseteq\Aut(\F_z)$. So the
inclusion map from $\Aut(\F)$ to $\Aut(\F_z)$ can be thought of as a
``restriction map''
\begin{eqnarray}
\label{E:forget}
\rho\colon \Aut(\F) \longrightarrow \Aut(\F_z)
\end{eqnarray}
given by remembering only that an automorphism preserves fusion in $\F_z$. We
want to make explicit in Lemma~\ref{L:linkres} that the map $\rho$ of
\eqref{E:forget} comes from a restriction map on the level of centric linking
systems. First we need to recall some information about the normalizer of $T_k$
in $\L$ and $\L_z$. 

\begin{lemma}
\label{L:Sselfnorm}
The following hold after identifying $T_k$ with its image $\delta_{T_k}(T_k)
\leq \Aut_\L(T_k)$. 
\begin{enumerate}
\item[\textup{(a)}] $\Aut_{\L_z}(T_k)$ is an extension of $T_k$ by $C_2 \times
S_4$, and $\Aut_{\L}(T_k)$ is an extension of $T_k$ by $C_2 \times GL_3(2)$ in
which the $GL_3(2)$ factor acts naturally on $T_k/\Phi(T_k)$. In each case, a
$C_2$ factor acts as inversion on $T_k$.  Also, $T_k$ is equal to its
centralizer in each of the above normalizers, $Z(\Aut_{\L_z}(T_k)) = Z(S)$, and
$Z(\Aut_\L(T_k)) = 1$. 
\item[\textup{(b)}] $\Aut_{\F}(S) = \Inn(S) = \Aut_{\F_z}(S)$ and $\Aut_{\L}(S)
= \delta_S(S) = \Aut_{\L_z}(S)$. 
\end{enumerate}
\end{lemma}
\begin{proof}
For part (a), see Lemma~4.3 and Proposition~5.4 of
\cite{AschbacherChermak2010}. 

Since $T_k$ is the unique abelian subgroup of its order in $S$ by
\cite[Lemma~4.9(c)]{AschbacherChermak2010}, it is characteristic. By the
uniqueness of restrictions (see Section~\ref{SS:restrictions}), we may
therefore view $\Aut_{\L}(S)$ as a subgroup of $\Aut_{\L}(T_k)$. Since
$\Aut_{\L}(T_k)$ has self-normalizing Sylow $2$-subgroups by (a), the same
holds for $\Aut_\L(S)$.  Now (b) follows for $\L$, and for $\F$ after applying
$\pi$. This also implies the statement for $\L_z$ and $\F_z$, as subcategories. 
\end{proof}

There is a $3$-dimensional commutative diagram related to \eqref{E:diagram}
that is the point of the next lemma.

\begin{lemma}
\label{L:linkres}
There is a restriction map $\hat{\rho}\colon \Aut(\L) \to
\Aut(\L_z)$, with kernel the automorphisms induced by conjugation by
$\delta_S(Z(S)) \leq \Aut_\L(S)$, which makes the diagram
\[
\xymatrix{
\Aut(\L) \ar[d]_{\widetilde{\mu}_{\L}} \ar[r]^{\hat{\rho}} & \Aut(\L_z) \ar[d]^{\widetilde{\mu}_{\L_z}}\\
\Aut(\F) \ar[r]_{\rho} & \Aut(\F_z)\,,
}
\]
commutative, which commutes with the conjugation maps out of
\[
\xymatrix{
\Aut_\L(S) \ar[d]_{\pi_S} \ar[r]^{\id} & \Aut_{\L_z}(S) \ar[d]^{\pi_S}\\
\Aut_\F(S) \ar[r]_{\id} & \Aut_{\F_z}(S), \\
}
\]
and which therefore induces a commutative diagram
\[
\xymatrix{
\Out(\L) \ar[d]_{\mu_\L} \ar[r]^{[\hat{\rho}]} & \Out(\L_z)
\ar[d]^{\mu_{\L_z}}\\ \Out(\F) \ar[r]_{[\rho]} & \Out(\F_z).\\
}
\]
\end{lemma}
\begin{proof}
Recall that we have arranged $\L_z \subseteq \L$. Thus, the horizontal maps in
the second diagram are the identity maps by Lemma~\ref{L:Sselfnorm}, and so the
lemma amounts to checking that an element of $\Aut(\L)$ sends
morphisms in $\L_z$ to morphisms in $\L_z$. For then, we can define its image
under $\hat{\rho}$ to have the same effect on objects, and to be the
restriction to $\L_z$ on morphisms.

Now fix an arbitrary $\alpha \in \Aut(\L)$, objects $P,Q \in \F^c =
\F_z^c$, and a morphism $\phi \in \Mor_{\L}(P,Q)$. Let $Z(S) = \gen{z}$.  By
two applications of Axiom (C) for a linking system (\S\S\ref{SS:axiomc}),
\begin{eqnarray}\label{E:inclz} 
\iota_{P,S}\circ \delta_{P}(z) = \delta_{S}(z)\circ \iota_{P,S} 
\text{\quad and \quad}
\iota_{\alpha(P),S} \circ \delta_{\alpha(P)}(z) = \delta_S(z) \circ
\iota_{\alpha(P),S}, 
\end{eqnarray}
because $\pi(\iota_{P,S})(z) = \pi(\iota_{\alpha(P),S})(z) = z$.  Since
$\alpha_S$ is an automorphism of $\delta_{S}(S) \cong S$, it sends
$\delta_S(z)$ to itself.  Thus, $\alpha$ sends the right side of the first
equation of \eqref{E:inclz} to the right side of the second, since it sends
inclusions to inclusions. Thus
\[
\iota_{\alpha(P),S} \circ \alpha(\delta_{P}(z)) = \iota_{\alpha(P),S} \circ
\delta_{\alpha(P)}(z). 
\]
However, each morphism in $\L$ is a monomorphism
\cite[Proposition~4]{Oliver2010}, so we obtain
\begin{eqnarray}
\label{E:alphaz}
\alpha(\delta_{P}(z)) = \delta_{\alpha(P)}(z),
\end{eqnarray}
and the same holds for $Q$ in place of $P$. 

Since $\phi \in \Mor(\L_z)$, we have $\pi(\phi)(z) = z$, so by two more
applications of Axiom (C),
\begin{eqnarray}
\label{E:phiz}
\phi \circ \delta_{P}(z) =  \delta_{Q}(z) \circ \phi
\text{\quad and \quad}
\alpha(\phi) \circ \delta_{\alpha(P)}(z) =
\delta_{\alpha(Q)}(\pi(\alpha(\phi))(z)) \circ \alpha(\phi). 
\end{eqnarray}
After applying $\alpha$ to the left side of the first equation of
\eqref{E:phiz}, we obtain the left side of the second by \eqref{E:alphaz}. 
Thus, comparing right sides, we obtain
\[
\delta_{\alpha(Q)}(z) \circ \alpha(\phi) =
\delta_{\alpha(Q)}(\pi(\alpha(\phi)(z))) \circ \alpha(\phi)
\]
Since each morphism in $\L$ is an epimorphism
\cite[Proposition~4]{Oliver2010}, it follows that 
\[
\delta_{\alpha(Q)}(z) = \delta_{\alpha(Q)}(\pi(\alpha(\phi))(z)).
\]
Hence, $\pi(\alpha(\phi))(z) = z$ because $\delta_{\alpha(Q)}$ is injective
(Axiom (A2)). That is, $\alpha(\phi) \in \Mor(\L_z)$ as required.

The kernel of $\hat{\rho}$ is described via a diagram chase in
\eqref{E:diagram}.  Suppose $\hat{\rho}(\alpha)$ is the identity.  Then,
$\alpha$ is sent to the identity automorphism of $S$ by $\widetilde{\mu}_\L$, since
$\rho$ is injective. Thus, $\alpha$ comes from a normalized $1$-cocycle by
\eqref{E:diagram} and these are in turn induced by elements of $Z(S)$ since
$\varprojlim^1(\Z_\F)$ is trivial \cite[Lemma~3.2]{LeviOliver2002}.
\end{proof}

\begin{lemma}
\label{L:littlelemma}
Let $G$ be a finite group and let $V$ be an abelian normal $2$-subgroup of $G$
such that $C_G(V)\leq V$. Let $\alpha$ be an automorphism of $G$ such that
$[V,\alpha]=1$.  Then $[G,\alpha] \leq V$, and if $G$ acts fixed point freely
on $V/\Phi(V)$ and $\alpha^2\in\Inn(G)$, then the order of $\alpha$ is at most
the exponent of $V$.
\end{lemma}
\begin{proof}
As $[V,\alpha]=1$, we have $[V,G,\alpha]\leq [V,\alpha]=1$ and
$[\alpha,V,G]=[1,G]=1$. Hence, by the Three subgroups lemma, it follows that
$[G,\alpha,V]=1$. As $C_G(V)\leq V$, this means 
\[
[G,\alpha]\leq V.
\] 
Assume from now on that $G$ acts fixed point freely on $V/\Phi(V)$. 
Write $G^*:= G \rtimes \gen{\alpha}$ for the semidirect product of $G$ by
$\gen{\alpha}$.  As $[V,\alpha]=1$ and $[G,\alpha]\leq V$, the subgroup
$W:=V\gen{\alpha}$ is an abelian normal subgroup of $G^*$ with $[W,G^*]\leq V$.

As $[V,\alpha]=1$, it follows that $[V,\alpha^2]=1$. So $\alpha^2\in\Inn(G)$ is
realized by conjugation with an element of $C_G(V)=V$. Pick $u\in V$ with
$\alpha^2=c_u|_G$. This means that, for any $g\in G$, we have ${
}^{u^{-1}\alpha^2}\!g=g$ in $G^*$. So $Z:=\gen{u^{-1}\alpha^2}$ centralizes $G$
in $G^*$. Since $W$ is abelian and contains $Z$, it follows that $Z$ lies in the
centre of $G^*=WG$. Set \[\ol{G^*}=G^*/Z.\] Because $C_G(V)\leq V$, the order of $u$
equals the order of $c_u|_G=\alpha^2$. Hence, $Z\cap G=1=Z\cap\gen{\alpha}$.
So  $|\ov{\alpha}|=|\alpha|$ and $G\cong \ov{G}$. In particular, we have
$\ov{V}\cong V$ and $\ov{G}$ acts fixed point freely on $\ov{V}/\Phi(\ov{V})$.
Note also that $\ov{\alpha}^2=\ov{u}$. Hence $|\ov{W}/\ov{V}|=2$ and
$\Phi(\ov{W})\leq \ov{V}$. Moreover, letting $n\in\mathbb{N}$ such that $2^n$
is the exponent of $V$, we have $|\alpha|=|\ov{\alpha}|\leq 2\cdot
2^n=2^{n+1}$. Assume $|\ov{\alpha}|=2^{n+1}$. Then $\ov{u}=\ov{\alpha}^2$ has
order $2^n$ and is thus not a square in $\ov{V}$. Note that
$\Phi(\ov{V})=\{v^2\colon v\in\ov{V}\}$ and $\Phi(\ov{W})=\{w^2\colon w\in
W\}=\gen{\ov{\alpha}^2}\Phi(\ov{V})\leq\ov{V}$. Hence,
$\Phi(\ov{W})/\Phi(\ov{V})$ has order $2$. As $\ov{G}$
normalizes $\Phi(\ov{W})/\Phi(\ov{V})$, it thus centralizes 
$\Phi(\ov{W})/\Phi(\ov{V})$ contradicting the assumption that
$\ov{G}$ acts fixed point freely on $\ov{V}/\Phi(\ov{V})$. Thus
$|\alpha|=|\ov{\alpha}|\leq 2^n$ which shows the assertion.
\end{proof}

We are now in a position to determine the automorphisms of $\F=\F_{\Sol}(q_l)$
and $\L=\L^c_{\Sol}(q_l)$. It is known that the field automorphisms induce
automorphisms of these systems as we will make precise next. Recall that the
field automorphism $\psi_5$ of $H$ of order $2^l$ normalizes $S$ and so
$\psi_5|_S$ is an automorphism of $\F_z=\F_S(H)$. By
\cite[Lemma~5.7]{AschbacherChermak2010}, the automorphism $\psi_5|_S$ is
actually also an automorphism of $\F$. We thus denote it by $\psi_{\F}$ and
refer to it as the field automorphism of $\F$ induced by $\psi_5$. By
Proposition~\ref{P:autspin}, this automorphism has order $2^l$.

\smallskip

By \cite[Proposition~3.3(d)]{LeviOliver2002}, there is a unique lift $\psi$ of
$\psi_\F$ under $\widetilde{\mu}_\L$ that is the identity on
$\pi^{-1}(\F_{\Sol}(5))$ and restricts to $\widetilde{\kappa}_H(\psi_5)$ on
$\L_z$. We refer to $\psi$ as the field automorphism of $\L$ induced by
$\psi_5$.

\begin{theorem}
\label{T:outsol}
Fix $l \geq 0$, and set $q_l = 5^{2^l}$ as before. The map $\mu_\L\colon
\Out(\L^c_{\Sol}(q_l)) \to \Out(\F_{\Sol}(q_l))$ is an isomorphism, and 
\[
\Out(\L^c_{\Sol}(q_l)) \cong \Out(\F_{\Sol}(q_l)) \cong C_{2^l}
\]
is induced by field automorphisms.  Also, the automorphism group
$\Aut(\L^c_{\Sol}(q_l))$ is a split extension of $S$ by
$\Out(\L^c_{\Sol}(q_l))$; in particular, it is a $2$-group. 

More precisely, if $\psi$ is the field automorphism of $\L^c_{\Sol}(q_l)$
induced by $\psi_5$, then $\psi$ has order $2^l$ and $\Aut(\L^c_{\Sol}(q_l))$
is the semidirect product of $\Aut_\L(S)\cong S$ with the cyclic group
generated by $\psi$.
\end{theorem}
\begin{proof}
We continue to write $\L = \L^c_{\Sol}(q_l)$, $\F = \F_{\Sol}(q_l)$, $\L_z =
\L^c_{\Spin}(q_l)$, and $\F_z = \F_{\Spin}(q_l)$, and we continue to assume
that $\L$ has been chosen so as to contain $\L_z$ as a linking subsystem.
Recall that $T_k \leq S$ is homocyclic of rank $3$ and exponent $2^{k} =
2^{l+2}$. 

We first check whether the outer automorphism of $\L_z$ induced by a diagonal
automorphism of $H$ extends to $\L$, and we claim that it doesn't.  A non-inner
diagonal automorphism of $H$ is induced by conjugation by an element $t$ of
$\bar{T}$ by \cite[Theorem~2.5.1(b)]{GLS3}.  Its class as an outer automorphism
has order $2$, so if necessary we replace $t$ by an odd power and assume that
$t^2 \in T_k$. Now $T_{k}$ consists of the elements of $\bar{T}$ of order
dividing $2^{k}$, so $t$ has order $2^{k+1}$ and induces an automorphism of $H$
of order at least $2^{k}$.  For ease of notation, we identify $T_k$ with
$\delta_{T_k}(T_k)\leq \Aut_\L(T_k)$, and we identify $s \in S$ with
$\delta_S(s) \in \Aut_\L(S)$.

Let $\tau = \widetilde{\kappa}_H([c_t]) \in \Aut(\L_z)$, and assume that
$\tau$ lifts to an element $\hat{\tau} \in \Aut(\L)$ under the map
$\hat{\rho}$ of Lemma~\ref{L:linkres}.  As $\hat{\rho}(\hat{\tau}) = \tau =
\widetilde{\kappa}_H(c_t)$, we have
$\hat{\rho}(\hat{\tau}^2)=\tau^2=\widetilde{\kappa}_H(c_{t^2})$, i.e.
$\hat{\rho}(\hat{\tau}^2)$ acts on every object and every morphism of
$\L_z=\L_S^c(H)$ as conjugation by $t^2$.  Similarly, if we take the
conjugation automorphism $c_{t^2}$ of $\L$ by $t^2$ (or more precisely the
conjugation automorphism $c_{\delta_S(t^2)}$ of $\L$ by $\delta_S(t^2)$), then
$\hat{\rho}(c_{t^2})$ is just the conjugation automorphism of $\L_z$ by $t^2$.
So according to the remark at the end of \S\S\ref{SS:linkaut}, the automorphism
$\hat{\rho}(c_{t^2})$ acts also on $\L_z$ via conjugation by $t^2$, which shows that
$\hat{\rho}(\hat{\tau}^2)=\hat{\rho}(c_{t^2})$.  By the description of the
kernel in Lemma~\ref{L:linkres}, we have $\hat{\tau}^2=c_{t^2}$ or
$\hat{\tau}^2=c_{t^2z}$. 

Now set $\alpha:=\hat{\tau}_{T_k} \in \Aut(\Aut_\L(T_k))$.  From what we have
shown, it follows that $\alpha$ equals the conjugation automorphism $c_{t^2}$
or $c_{t^2z}$ of $\Aut_\L(T_k)$. Note that $|c_{t^2z}| = |t^2z| = |t^2| =
|c_{t^2}|$, since $Z(\Aut_\L(T_k))=1$ by Lemma~\ref{L:Sselfnorm}(a). Hence, 
\begin{eqnarray}
\label{E:orderalpha}
|\alpha| = 2|\alpha^2| = 2|t^2| = 2^{k+1},
\end{eqnarray}
On the other hand, $\alpha$ centralizes $T_k$, and we have seen that 
$\alpha^2$ is an inner automorphism of $\Aut_\L(T_k)$.  Moreover, by
Lemma~\ref{L:Sselfnorm}(a), $C_{\Aut_\L(T_k)}(T_k)=T_k$, and $\Aut_\L(T_k)$
acts fixed point freely on $T_k/\Phi(T_k)$. The hypotheses of
Lemma~\ref{L:littlelemma} thus hold for $G=\Aut_\L(T_k)$ and $\alpha \in
\Aut(G)$. So $\alpha$ has order at most $2^{k}$ by that lemma, contradicting
\eqref{E:orderalpha}.  We conclude that a diagonal automorphism of $\L_z$ does
not extend to an automorphism of $\L$.  

The existence of the field automorphism $\psi_{\F}$ of $\F$ and the fact that
$\psi_\F$ has order $2^l$ now yield together with Proposition~\ref{P:autspin}
that $\Out(\F) \cong C_{2^l}$ is generated by the image of $\psi_\F$ in
$\Out(\F)$. Moreover, by \cite[Lemma~3.2]{LeviOliver2002} and the exactness of
the third column of \eqref{E:diagram}, the maps $\mu_\L$ and $\mu_{\L_z}$ are
isomorphisms. Thus, 
\[
\Out(\L) \cong \Out(\F) \cong C_{2^l}.  
\]

Let $\psi$ be the field automorphism of $\L$ induced by $\psi_5$ as above.
Then $\psi$ is the identity on $\pi^{-1}(\F_{\Sol}(5))$ by definition.  It
remains to show that $\psi$ has order $2^l$, since this will imply that
$\Aut(\L)$ is a split extension of $\Aut_\L(S) \cong S$ by
$\langle\psi\rangle\cong \Out(\L) \cong \Out(\F)$. 

The automorphism $\psi^{2^l}$ maps to the trivial automorphism of $\F$, and so
is conjugation by an element of $Z(S)$ by \eqref{E:diagram}.  Now $\psi^{2^l}$
is trivial on $\Aut_{\L^c_{\Sol}(5)}(\Omega_2(T_k))$, whereas $z \notin
Z(\Aut_{\L^c_{\Sol}(5)}(\Omega_2(T_k)))$ by Lemma~\ref{L:Sselfnorm}(a) as $T_2
= \Omega_2(T_k)$ is the torus of $\L^c_{\Sol}(5)$. Thus, since a morphism
$\phi$ is fixed by $c_z$ if and only if $\pi(\phi)(z) = z$ (Axiom (C)), we
conclude that $\psi^{2^l}$ is the identity automorphism of $\L$, and this
completes the proof.
\end{proof}

\section{Extensions}\label{S:ext}

In this section, we recall a result of Linckelmann on the Schur multipliers of
the Benson-Solomon systems, and we prove that each saturated fusion system $\F$
with $F^*(\F) \cong \F_{\Sol}(q_l)$ is a split extension of $F^*(\F)$ by a
group of outer automorphisms. 

Recall that the \emph{hyperfocal subgroup} of a saturated $p$-fusion system
$\F$ over $S$ is defined to be the subgroup of $S$ given by
\[
\hyp(\F) = \gen{[\phi, s] := \phi(s)s^{-1} \mid s \in P \leq S \text{ and }
\phi \in O^p(\Aut_\F(P))}. 
\]
A subsystem $\F_0$ over $S_0 \leq S$ is said to be of \emph{$p$-power index} in
$\F$ if $\hyp(\F) \leq S_0$ and $O^p(\Aut_\F(P)) \leq \Aut_{\F_0}(P)$ for each
$P \leq S_0$.  There is always a unique normal saturated subsystem on
$\hyp(\F)$ of $p$-power index in $\F$, which is denoted by $O^p(\F)$ \cite[\S
I.7]{AschbacherKessarOliver2011}. We will need the next lemma in
\S\S\ref{SS:exttop}.

\begin{lemma}\label{L:hyperfocal}
Let $\F$ be a saturated fusion system over $S$, and let $\F_0$ be a weakly
normal subsystem of $\F$ over $S_0\leq S$. Assume that $O^p(\Aut_\F(S_0))\leq
\Aut_{\F_0}(S_0)$. Then $O^p(\Aut_\F(P))\leq \Aut_{\F_0}(P)$ for every $P \leq
S_0$. Thus, if in addition $\hyp(\F) \leq S_0$, then $\F_0$ has $p$-power index
in $\F$.  
\end{lemma}

\begin{proof}
Note that $\Aut_{\F_0}(P)$ is normal in $\Aut_\F(P)$ for every $P\leq S_0$,
since $\F_0$ is weakly normal in $\F$. We need to show that
$\Aut_\F(P)/\Aut_{\F_0}(P)$ is a $p$-group for every $P\leq S_0$. Suppose this
is false and let $P$ be a counterexample of maximal order. Our assumption gives
$P<S_0$. Hence, $P<Q:=N_{S_0}(P)$, and the maximality of $P$ implies that 
\[ 
\Aut_\F(Q)/\Aut_{\F_0}(Q)
\]
is a $p$-group. Notice that 
\begin{eqnarray*}
N_{\Aut_\F(Q)}(P)/N_{\Aut_{\F_0}(Q)}(P) &\cong& N_{\Aut_{\F}(Q)}(P)\Aut_{\F_0}(Q)/\Aut_{\F_0}(Q)\\ 
&\leq& \Aut_\F(Q)/\Aut_{\F_0}(Q)
\end{eqnarray*}
and thus $N_{\Aut_\F(Q)}(P)/N_{\Aut_{\F_0}(Q)}(P)$ is a $p$-group. 

If $\alpha\in\Hom_\F(P,S)$ with $\alpha(P)\in\F^f$ then conjugation by $\alpha$
induces a group isomorphism from $\Aut_\F(P)$ to $\Aut_\F(\alpha(P))$. As
$\F_0$ is weakly normal, we have $\alpha(P)\leq S_0$ and conjugation by
$\alpha$ takes $\Aut_{\F_0}(P)$ to $\Aut_{\F_0}(\alpha(P))$. So upon replacing
$P$ by $\alpha(P)$, we may assume without loss of generality that $P$ is fully
$\F$-normalized. Then $P$ is also fully $\F_0$-normalized by
\cite[Lemma~3.4(5)]{AschbacherNormal}. By the Sylow axiom, $\Aut_{S_0}(P)$ is a
Sylow $p$-subgroup of $\Aut_{\F_0}(P)$. So the Frattini argument yields
\[
\Aut_\F(P)=\Aut_{\F_0}(P)N_{\Aut_\F(P)}(\Aut_{S_0}(P))
\]
and thus 
\[
\Aut_\F(P)/\Aut_{\F_0}(P)\cong
N_{\Aut_\F(P)}(\Aut_{S_0}(P))/N_{\Aut_{\F_0}(P)}(\Aut_{S_0}(P)).
\]  
By the extension axiom for $\F$ and $\F_0$, each element of
$N_{\Aut_\F(P)}(\Aut_{S_0}(P))$ extends to an automorphism of $\Aut_\F(Q)$, and
each element of $N_{\Aut_{\F_0}(P)}(\Aut_{S_0}(P))$ extends to an automorphism
of $\Aut_{\F_0}(Q)$. Therefore, the map 
\[
\Phi\colon N_{\Aut_\F(Q)}(P)\rightarrow
N_{\Aut_\F(P)}(\Aut_{S_0}(P)),\;\phi\mapsto \phi|_P
\]
is an epimorphism which maps $N_{\Aut_{\F_0}(Q)}(P)$ onto
$N_{\Aut_{\F_0}(P)}(\Aut_{S_0}(P))$. Hence, 
\begin{eqnarray*}
\Aut_\F(P)/\Aut_{\F_0}(P)&\cong&
N_{\Aut_\F(P)}(\Aut_{S_0}(P))/N_{\Aut_{\F_0}(P)}(\Aut_{S_0}(P))\\ &\cong&
N_{\Aut_\F(Q)}(P)/N_{\Aut_{\F_0}(Q)}(P)\ker(\Phi).
\end{eqnarray*}
We have seen above that $N_{\Aut_\F(Q)}(P)/N_{\Aut_{\F_0}(Q)}(P)$ is a
$p$-group, and therefore also
\[
N_{\Aut_\F(Q)}(P)/N_{\Aut_{\F_0}(Q)}(P)\ker(\Phi)
\] 
is a $p$-group. Hence, $\Aut_\F(P)/\Aut_{\F_0}(P)$ is a $p$-group, and this
contradicts our assumption that $P$ is a counterexample.
\end{proof}

\subsection{Extensions to the bottom}
A \emph{central extension} of a fusion system $\F_0$ is a fusion system $\F$
such that $\F/Z \cong \F_0$ for some subgroup $Z \leq Z(\F)$. The central
extension is said to be \emph{perfect} if $\F = O^p(\F)$. Linckelmann has shown
that the Schur multiplier of a Benson-Solomon system is trivial.

\begin{theorem}[Linckelmann]
\label{T:schurmult}
Let $\F$ be a perfect central extension of a Benson-Solomon fusion system
$\F_0$. Then $\F = \F_0$. 
\end{theorem}
\begin{proof}
This follows from Corollary~4.4 of \cite{Linckelmann2006b} together with the
fact that $\Spin_7(q)$ has Schur multiplier of odd order when $q$ is odd
\cite[Tables~6.1.2, 6.1.3]{GLS3}. 
\end{proof}

\subsection{Extensions to the top}\label{SS:exttop}

The next theorem describes the possible extensions $(S,\F)$ of a Benson-Solomon
system $(S_0,\F_0)$.  The particular hypotheses are best stated in terms of the
generalized Fitting subsystem of Aschbacher \cite{AschbacherGeneralized}, but
they are equivalent to requiring that $\F_0 \norm \F$ and $C_S(\F_0) \leq S_0$,
where $C_S(\F_0)$ is the centralizer constructed in \cite[\S
6]{AschbacherGeneralized}. This latter formulation is sometimes expressed by
saying that $\F_0$ is \emph{centric normal} in $\F$. 

\begin{theorem}
\label{T:solext}
Let $l$ be any nonnegative integer, and let $\F_0 = \F_{\Sol}(5^{2^l})$ be a
Benson-Solomon system over the $2$-group $S_0$.
\begin{enumerate}
\item[\textup{(a)}] If $\F$ is a saturated fusion system over $S$ such that
$F^*(\F) = \F_0$, then $\F_0 = O^2(\F)$, $S$ splits over $S_0$, and the map
$S/S_0 \to \Out(\F_0)$ induced by conjugation is injective.
\item[\textup{(b)}] Conversely, given a subgroup of $A \leq \Out(\F_0) \cong
C_{2^l}$, there is a saturated fusion system $\F$ over a $2$-group $S$ such
that $F^*(\F) = \F_0$ and the map $S/S_0 \to \Out(\F_0)$ induced by conjugation
on $S_0$ has image $A$.  Moreover, the pair $(S,\F)$ with these properties is
uniquely determined up to isomorphism.

If $\L_0$ is a centric linking system associated to $\F_0$, then
$\Aut_{\L_0}(S_0)=S_0$, and the $p$-group $S$ can be chosen to be the preimage
of $A$ in $\Aut(\L_0)$ under the quotient map from $\Aut(\L_0)$ to
$\Out(\L_0)\cong \Out(\F_0)$. 
\end{enumerate}
\end{theorem}

\begin{proof}
Let $\F$ be a saturated fusion system over $S$ such that $F^*(\F) = \F_0$.  Set
$\F_1 = \F_0S$, the internal extension of $\F_0$ by $S$, as in \cite{Henke2013}
or \cite[\S 8]{AschbacherGeneralized}. According to
\cite[Proposition~1.31]{AOV2012}, there is a normal pair of linking systems
$\L_0 \norm \L_1$, associated to the normal pair $\F_0 \norm \F_1$.
Furthermore, $\L_0 \unlhd \L_1$ can be chosen such that $\L_0$ is a centric
linking system. There is a natural map from $\Aut_{\L_1}(S_0)$ to $\Aut(\L_0)$
which sends a morphism $\phi\in\Aut_{\L_1}(S_0)$ to conjugation by $\phi$. (So
the restriction of this map to $\Aut_{\L_0}(S_0)$ is the conjugation map
described in \S\S\ref{SS:linkaut}.)

The centralizer $C_S(\F_0)$ depends a priori on the fusion system $\F$, but it
is shown in  \cite[Lemma~1.13]{Lynd2015} that it does not actually matter
whether we form $C_S(\F_0)$ inside of $\F$ or inside of $\F_1$. Moreover, since
$F^*(\F) = \F_0$, it follows from \cite[Theorem~6]{AschbacherGeneralized} that
$C_S(\F_0)=Z(\F_0) = 1$. Thus, by a result of Semeraro
\cite[Theorem~A]{Semeraro2015}, the conjugation map
$\Aut_{\L_1}(S_0)\xrightarrow{\operatorname{conj}} \Aut(\L_0)$ is injective. By
Lemma~\ref{L:Sselfnorm}, we have $S_0=\Aut_{\L_0}(S_0)$ via the inclusion
functor $\delta_1$ for $\L_1$. By Theorem~\ref{T:outsol}, $\Aut(\L_0)$ is a
$2$-group which splits over $S_0$. Moreover, by the same theorem, we have that
$C_{\Aut(\L_0)}(S_0)\leq S_0$ and $\Out(\L_0)\cong\Out(\F_0)$ is cyclic. Since
$(\delta_1)_{S_0}(S)\cong S$ is a Sylow $2$-subgroup of $\Aut_\L(S_0)$ by
\cite[Proposition~4(d)]{Oliver2010}, we can conclude that 
\[
S_0 = \Aut_{\L_0}(S_0) \norm \Aut_{\L_1}(S_0) = S,
\]
via the inclusion functor $\delta_1$ for $\L_1$. Moreover, it follows that $S$
splits over $S_0$, and $C_S(S_0)\leq S_0$. The latter property means that the
map
\[
S/S_0 \rightarrow \Out(\F_0)
\]
is injective. In particular, $S/S_0$ is cyclic as $\Out(\F_0)$ is cyclic. 

Next, we show that $O^2(\F) = \F_0$. Fix a subgroup $P \leq S$, and let $\alpha
\in \Aut_\F(P)$ be an automorphism of odd order. Then $\alpha$ induces an
odd-order automorphism of the cyclic $2$-group $P/(P \cap S_0) \cong
PS_0/S_0\leq S/S_0$.  This automorphism must be trivial, and so $[P, \alpha]
\leq S_0$.  Hence, $[P, O^2(\Aut_\F(P))] \leq S_0$ for all $P \leq S$. Since
$\hyp(\F_0) = S_0$, we have $\hyp(\F) = S_0$. Note that $\Aut_\F(S_0)$ is a
$2$-group as $\Aut_\F(S_0)\leq \Aut(\F_0)$ and $\Aut(\F_0)$ is a $2$-group by
Theorem~\ref{T:outsol}. Therefore $O^2(\F) = \F_0$ by Lemma~\ref{L:hyperfocal}.
We conclude that $\F_1 = \F$ by the uniqueness statement in
\cite[Theorem~1]{Henke2013}. This completes the proof of (a). Moreover, we have
seen that the following property holds for any normal pair $\L_0\unlhd\L$
attached to $\F_0\unlhd\F$:
\begin{eqnarray}\label{E:Conj}
S_0 = \Aut_{\L_0}(S_0) \norm \Aut_{\L_1}(S_0) = S \mbox{\qquad and \qquad
}S\xrightarrow{\;\text{conj}\;}\Aut(\L_0)\mbox{ is injective.}
\end{eqnarray}
Finally, we prove (b).  Fix a centric linking system $\L_0$ associated to
$\F_0$ with inclusion functor $\delta_0$. Let $S \leq \Aut(\L_0)$ be the
preimage of $A$ under the quotient map to $\Out(\F_0)$. We will identify $S_0$
with $\delta_0(S_0)$ so that $S_0=\Aut_{\L_0}(S_0)$ by Lemma~\ref{L:Sselfnorm}.
Write $\iota\colon S_0\rightarrow \Aut(\L_0),s\mapsto c_s$ for map sending
$s\in S_0$ to the automorphism of $\L_0$ induced by conjugation with $s$ in
$\L_0$. Then $\iota(S_0)$ is normal in $S$. Let $\chi\colon S\rightarrow
\Aut(S_0)$ be the map defined by $\alpha\mapsto \iota^{-1}\circ
c_\alpha|_{\iota(S_0)}\circ \iota$; i.e. $\chi$ corresponds to conjugation in
$S$ if we identify $S_0$ with $\iota(S_0)$. We argue next that the following
diagram commutes:
\begin{eqnarray}\label{D:SquareTriangle}
\vcenter{
\xymatrix{
S_0 \ar[r]^{\iota\;\;\;}\ar[d]_{\iota} & \Aut(\L_0)\ar[d]^{\alpha\mapsto \alpha_{S_0}}\\
S \ar[ur]^{\operatorname{incl}}\ar[r]^{\chi\;\;\;} & \Aut(S_0)
}}
\end{eqnarray}
The upper triangle clearly commutes. Observe that  $\alpha\circ \iota(s)\circ
\alpha^{-1}=\alpha\circ c_s\circ
\alpha^{-1}=c_{\alpha_{S_0}(s)}=\iota(\alpha_{S_0}(s))$ for every $s\in S_0$
and $\alpha\in S$. Hence, for every $\alpha\in S$ and $s\in S_0$, we have
$(\iota^{-1}\circ c_\alpha|_{\iota(S_0)}\circ \iota)(s)=\iota^{-1}(\alpha\circ
\iota(s)\circ \alpha^{-1})=\alpha_{S_0}(s)$ and so the lower triangle commutes. 

We will now identify $S_0$ with its image in $S$ under $\iota$, so that $\iota$
becomes the inclusion map and $\chi$ corresponds to the map $S\rightarrow
\Aut(S_0)$ induced by conjugation in $S$. As the above diagram commutes, it
follows then that the diagram in \cite[Theorem~9]{Oliver2010} commutes when we
take $\Gamma = S$ and $\tau\colon S\rightarrow \Aut(\L_0)$ to be the inclusion.
Thus, by that theorem, there is a saturated fusion system $\F$ over $S$ in
which $\F_0$ is weakly normal, and there is a corresponding normal pair of
linking systems $\L_0 \norm \L$ (in the sense of \cite[\S 1.5]{AOV2012}) such
that $S = \Aut_{\L}(S_0)$ has the given action on $\L_0$ (i.e. the automorphism
of $\L_0$ induced by conjugation with $s\in S$ in $\L$ equals the automorphism
$s$ of $\L_0$).  By the same theorem, the pair $(\F,\L)$ is unique up to
isomorphism of fusion systems and linking systems with these properties. Since
$\F_0$ is simple \cite{Linckelmann2006}, $\F_0$ is in fact normal in $\F$ by a
result of Craven \cite[Theorem~A]{Craven2011}.  Thus, since $C_S(\F_0) \leq
C_S(S_0) \leq S_0$, it is a consequence of \cite[(9.1)(2),
(9.6)]{AschbacherGeneralized} that $F^*(\F) = \F_0$.

So it remains only to prove that $(S,\F)$ is uniquely determined up to an
isomorphism of fusion systems. Let $\F'$ be a saturated fusion system over a
$2$-group $S'$ such that $F^*(\F')=\F_0$, and such that the map
$S'/S_0\rightarrow \Out(\F_0)$ induced by conjugation has image $A$. Then by
(a), $\F_0=O^2(\F')$. So by  \cite[Proposition~1.31]{AOV2012}, there is a
normal pair of linking systems $\L_0' \norm \L'$ associated to the normal 
pair $\F_0 \norm \F'$. Moreover, we can choose $\L_0'$ to be a centric linking
system. Since a centric linking system attached to $\F_0$ is unique, there is
an isomorphism $\theta\colon \L_0'\rightarrow \L_0$ of linking systems. We may
assume that the set of morphisms which lie in $\L'$ but not in $\L_0'$ is
disjoint from the set of morphisms in $\L_0$. Then we can construct a new
linking system from $\L'$ by keeping every morphism of $\L'$ which is not in
$\L_0'$ and replacing every morphism $\psi$ in $\L_0'$ by $\theta(\psi)$, and
then carrying over the structure of $\L'$ in the natural way. Thereby we may
assume $\L_0'=\L_0$. So we are given a normal pair $\L_0\unlhd \L'$ attached to
$\F_0\unlhd \F'$. By (\ref{E:Conj}) applied with $\L'$ and $\F'$ in place of
$\F$ and $\L$, we have $S_0=\Aut_{\L_0}(S_0)\unlhd \Aut_{\L'}(S_0)=S'$ via the
inclusion functor $\delta'$ of $\L'$. Let    
\[
\tau\colon S'\rightarrow \Aut(\L_0)
\]
be the map taking $s\in S'$ to the automorphism of $\L_0$ induced by
conjugation with $s$ in $\L'$. Again using (\ref{E:Conj}), we see that $\tau$
is injective. Note also that $\tau$ restricts to the identity on $S_0$ if we
identify $S_0$ with $\iota(S_0)$ as above. Recall that the map
$S'/S_0\rightarrow \Out(\F_0)$ induced by conjugation has image $A$. So
Theorem~\ref{T:outsol} implies $\tau(S')=S$, i.e. we can regard $\tau$ as an
isomorphism $\tau\colon S'\rightarrow S$. So replacing $(S',\F')$ by
$(S,{}^\tau\!\F')$ and then choosing $\L_0\unlhd\L'$ as before, we may assume
$S=S'$. So $\F'$ is a fusion system over $S$ with $\F_0\unlhd \F'$, and
$\L_0\unlhd \L'$ is a normal pair of linking systems associated to $\F_0\unlhd
\F'$ such that $\Aut_{\L'}(S_0)=S$ via $\delta'$.  Let $s\in S$. Recall that
$\tau(s)$ is the automorphism of $\L_0$ induced by conjugation with $s$ in
$\L'$. Observe that the automorphism of $S_0=\Aut_{\L_0}(S_0)$ induced by
$\tau(s)$ equals just the automorphism of $S_0$ induced by conjugation with $s$
in $S$. Similarly, the automorphism $s$ of $\L_0$ equals the automorphism of
$\L_0$ given by conjugation with $s$ in $\L$, and so induces on
$S_0=\Aut_{\L_0}(S_0)$ just the automorphism given by conjugation with $s$ in
$S$. Theorem~\ref{T:outsol} gives $C_{\Aut(\L_0)}(S_0)\leq S_0$ and this
implies that any two automorphisms of $\L_0$, which induce the same
automorphism on $S_0$, are equal. Hence, $\tau(s)=s$ for any $s\in S$. In other
words, $S=\Aut_{\L'}(S_0)$ induces by conjugation in $\L'$ the canonical action
of $S$ on $\L_0$. The uniqueness of the pair $(\F,\L)$ implies now $\F'\cong
\F$ and $\L'\cong \L$. This shows that $(S,\F)$ is uniquely determined up to
isomorphism. 
\end{proof}

\bibliographystyle{amsalpha}{}
\bibliography{/home/cpsmth/s05jl6/work/math/research/mybib}

\def\cprime{$'$}
\providecommand{\bysame}{\leavevmode\hbox to3em{\hrulefill}\thinspace}
\providecommand{\MR}{\relax\ifhmode\unskip\space\fi MR }
\providecommand{\MRhref}[2]{%
  \href{http://www.ams.org/mathscinet-getitem?mr=#1}{#2}
}
\providecommand{\href}[2]{#2}
\begin{thebibliography}{BMO16}

\bibitem[AC10]{AschbacherChermak2010}
Michael Aschbacher and Andrew Chermak, \emph{A group-theoretic approach to a
  family of 2-local finite groups constructed by {L}evi and {O}liver}, Ann. of
  Math. (2) \textbf{171} (2010), no.~2, 881--978.

\bibitem[AKO11]{AschbacherKessarOliver2011}
Michael Aschbacher, Radha Kessar, and Bob Oliver, \emph{Fusion systems in
  algebra and topology}, London Mathematical Society Lecture Note Series, vol.
  391, Cambridge University Press, Cambridge, 2011. \MR{2848834}

\bibitem[AOV12]{AOV2012}
Kasper K.~S. Andersen, Bob Oliver, and Joana Ventura, \emph{Reduced, tame and
  exotic fusion systems}, Proc. Lond. Math. Soc. (3) \textbf{105} (2012),
  no.~1, 87--152. \MR{2948790}

\bibitem[Asc08]{AschbacherNormal}
Michael Aschbacher, \emph{Normal subsystems of fusion systems}, Proc. Lond.
  Math. Soc. (3) \textbf{97} (2008), no.~1, 239--271. \MR{2434097
  (2009e:20044)}

\bibitem[Asc11]{AschbacherGeneralized}
\bysame, \emph{The generalized {F}itting subsystem of a fusion system}, Mem.
  Amer. Math. Soc. \textbf{209} (2011), no.~986, vi+110. \MR{2752788}

\bibitem[Ben98]{Benson1998c}
David~J. Benson, \emph{Cohomology of sporadic groups, finite loop spaces, and
  the {D}ickson invariants}, Geometry and cohomology in group theory ({D}urham,
  1994), London Math. Soc. Lecture Note Ser., vol. 252, Cambridge Univ. Press,
  Cambridge, 1998, pp.~10--23. \MR{1709949 (2001i:55017)}

\bibitem[BLO03]{BrotoLeviOliver2003}
Carles Broto, Ran Levi, and Bob Oliver, \emph{The homotopy theory of fusion
  systems}, J. Amer. Math. Soc. \textbf{16} (2003), no.~4, 779--856
  (electronic).

\bibitem[BMO16]{BrotoMollerOliver2016}
Carles Broto, Jesper~M. M{\o}ller, and Bob Oliver, \emph{Automorphisms of
  fusion systems of finite simple groups of lie type}, preprint (2016),
  arXiv:1601.04566.

\bibitem[COS08]{ChermakOliverShpectorov2008}
Andrew Chermak, Bob Oliver, and Sergey Shpectorov, \emph{The linking systems of
  the {S}olomon 2-local finite groups are simply connected}, Proc. Lond. Math.
  Soc. (3) \textbf{97} (2008), no.~1, 209--238. \MR{2434096 (2009g:55018)}

\bibitem[Cra11]{Craven2011}
David~A. Craven, \emph{Normal subsystems of fusion systems}, J. Lond. Math.
  Soc. (2) \textbf{84} (2011), no.~1, 137--158. \MR{2819694}

\bibitem[DW93]{DwyerWilkerson1993}
W.~G. Dwyer and C.~W. Wilkerson, \emph{A new finite loop space at the prime
  two}, J. Amer. Math. Soc. \textbf{6} (1993), no.~1, 37--64. \MR{1161306}

\bibitem[GL83]{GorensteinLyons1983}
Daniel Gorenstein and Richard Lyons, \emph{The local structure of finite groups
  of characteristic {$2$} type}, Mem. Amer. Math. Soc. \textbf{42} (1983),
  no.~276, vii+731. \MR{690900}

\bibitem[GLS98]{GLS3}
Daniel Gorenstein, Richard Lyons, and Ronald Solomon, \emph{The classification
  of the finite simple groups. {N}umber 3. {P}art {I}. {C}hapter {A}},
  Mathematical Surveys and Monographs, vol.~40, American Mathematical Society,
  Providence, RI, 1998, Almost simple $K$-groups. \MR{1490581 (98j:20011)}

\bibitem[Hen13]{Henke2013}
Ellen Henke, \emph{Products in fusion systems}, J. Algebra \textbf{376} (2013),
  300--319. \MR{3003728}

\bibitem[Lin06a]{Linckelmann2006b}
Markus Linckelmann, \emph{A note on the {S}chur multiplier of a fusion system},
  J. Algebra \textbf{296} (2006), no.~2, 402--408.

\bibitem[Lin06b]{Linckelmann2006}
\bysame, \emph{Simple fusion systems and the {S}olomon 2-local groups}, J.
  Algebra \textbf{296} (2006), no.~2, 385--401.

\bibitem[LO02]{LeviOliver2002}
Ran Levi and Bob Oliver, \emph{Construction of 2-local finite groups of a type
  studied by {S}olomon and {B}enson}, Geom. Topol. \textbf{6} (2002), 917--990
  (electronic).

\bibitem[LO05]{LeviOliver2005}
\bysame, \emph{Correction to: ``{C}onstruction of 2-local finite groups of a
  type studied by {S}olomon and {B}enson'' [{G}eom. {T}opol. {\bf 6} (2002),
  917--990 (electronic); mr1943386]}, Geom. Topol. \textbf{9} (2005),
  2395--2415 (electronic).

\bibitem[Lyn15]{Lynd2015}
Justin Lynd, \emph{A characterization of the 2-fusion system of {$L_4(q)$}}, J.
  Algebra \textbf{428} (2015), 315--356. \MR{3314296}

\bibitem[Oli10]{Oliver2010}
Bob Oliver, \emph{Extensions of linking systems and fusion systems}, Trans.
  Amer. Math. Soc. \textbf{362} (2010), no.~10, 5483--5500. \MR{2657688
  (2011f:55032)}

\bibitem[Oli16]{OliverReductions}
\bysame, \emph{Reductions to simple fusion systems}, Bulletin of the London
  Mathematical Society \textbf{48} (2016), no.~6, 923--934.

\bibitem[Sem15]{Semeraro2015}
Jason Semeraro, \emph{Centralizers of subsystems of fusion systems}, J. Group
  Theory \textbf{18} (2015), no.~3, 393--405. \MR{3341522}

\bibitem[Ste60]{Steinberg1960}
Robert Steinberg, \emph{Automorphisms of finite linear groups}, Canad. J. Math.
  \textbf{12} (1960), 606--615. \MR{0121427}

\end{thebibliography}

\end{document}